\documentclass[english,10pt]{article}
\usepackage[english]{babel}
\usepackage{amsmath,amsthm,amssymb}
\usepackage{latexsym}
\usepackage{enumerate}
\usepackage{hyperref}

\usepackage{color}

\newtheorem{theorem}{Theorem}
\newtheorem{lemma}{Lemma}

\newtheorem{corollary}{Corollary}
\theoremstyle{remark}

\title{Uniqueness and statistical properties of the Gibbs state on general one-dimensional lattice systems with markovian structure}

\author{Victor Vargas\thanks{Centre for Mathematics of the University of Porto, Porto-Portugal, vavargascu@gmail.co. Supported by FCT through the project UIDP/00144/2020.}}

\begin{document}

\maketitle

{\bf Abstract}
Let $M$ be a compact metric space and $X = M^{\mathbb{N}}$, we consider a set of admissible sequences $X_{A, I} \subset X$ determined by a continuous admissibility function $A : M \times M \to \mathbb{R}$ and a compact set $I \subset \mathbb{R}$. Given a Lipschitz continuous potential $\varphi : X_{A, I} \to \mathbb{R}$, we prove uniqueness of the Gibbs state $\mu_\varphi$ and we show that it is a Gibbs-Bowen measure and satisfies a central limit theorem.

\vspace{2mm}


\vspace{2mm}


\vspace{2mm}

\section{Introduction}

Typical problems, like existence of Gibbs states and their relation with the so called equilibrium states and maximizing measures, besides some problems of selection and non-selection at zero temperature on infinite spin one-dimensional lattice systems were deeply understood (see for instance \cite{MR2864625, MR3377291}). In the seminal works, the main interest of the authors was to characterize the behavior of the ``free energy'' associated to a suitable potential at lower temperatures (see \cite{zbMATH05995684, MR2316198} for details). In here, we are concerned with the uniqueness and some statistical properties of the Gibbs state associated to a normalized Lipschitz continuous potential defined on a general one-dimensional lattice system with markovian structure (see \cite{MR3194082, zbMATH07249339, zbMATH07434854} for details about the construction of the configuration space). Moreover, the techniques that we use here allow to prove our results looking only at the behavior of the dual Ruelle operator on the space of Borel probability measures.

The kind of model which we are dealing with was presented in \cite{MR2864625, MR3377291} for the context of Bernoulli systems. In that works, the authors studied some properties of the Gibbs states at lower temperatures and the relation between them and the so-called DLR Gibbs measures (better understood in the context of statistical mechanics). Some time later, in \cite{MR3194082}, a markovian approach of this model has been proposed where some of the sequences defined on the configuration space are forbidden through an admissibility function with finite memory. In \cite{zbMATH07434854}, the existence of equilibrium states through techniques of the Ruelle-Perron-Frobenius theory and some interesting connections between that class of infinite spin one dimensional lattice systems with the countable Markov shifts were presented. Moreover, the existence of accumulation points at zero temperature for a class of countable Markov shifts was showed, in a different context to the ones presented in \cite{MR3864383, MR2151222}. Besides that, in \cite{zbMATH07249339} some properties of the quadratic pressure and the cosine potential on this dynamical setting were studied.

In this work, following a similar approach to the one proposed in \cite{zbMATH06717808} (see also \cite{MR2478676} for a stochastic equations approach), we show existence of a Wasserstein metric contracting the dual Ruelle operator associated to a normalized potential which guarantees existence and uniqueness of a fixed point (the so called Gibbs state). Actually, in our approach it is not necessary to look at the spectral behavior of the Ruelle operator itself. Therefore, the results presented here are independent of the ones obtained via Ruelle-Perron-Frobenius theory in \cite{zbMATH07434854}. In addition, we show that such a fixed point is Gibbs in the sense of Bowen (in the sense of the one proposed in \cite{MR1860762}), and we prove that the Gibbs state has exponential decay of correlations which implies the mixing property. Furthermore, adapting some techniques proposed in \cite{Via97} (see also \cite{zbMATH03337280} for a seminal approach), we prove that the Gibbs state satisfies a central limit theorem.

The paper is organized as follows. In section \ref{preliminaries-section} we present the notation and the main definitions to be used throughout the paper. In section \ref{contraction-section} we prove Theorem \ref{theorem-FP} which guarantees the uniqueness of the Gibbs state. In section \ref{statistical-properties-section} we present the proofs of Theorem \ref{theorem-GB}, Theorem \ref{theorem-DC} and Theorem \ref{theorem-CL} which show interesting statistical properties of the Gibbs state.

\section{Preliminaries}
\label{preliminaries-section}

Let $(Y, d)$ be a compact metric space and $T : Y \to Y$ be a continuous endomorphism. Denote by $\mathrm{C}(Y)$, resp. $\mathrm{Lip}(Y)$, resp. $\mathcal{M}(Y) = \mathrm{C}^*(Y)$, resp. $\mathcal{M}_1(Y)$ and resp. $\mathcal{M}_T(Y)$ for the set of {\bf continuous functions}, resp. {\bf Lipschitz continuous functions}, resp. {\bf finite Borel measures}, resp. {\bf Borel probability measures} and resp. {\bf T-invariant Borel probability measures} on $Y$. As usual, we use the notation $\mathcal{B}_Y$ for the collection of {\bf Borel sets} on $Y$ and $\mathrm{Lip}(\varphi)$ for the {\bf Lipschitz constant} of $\varphi \in \mathrm{Lip}(Y)$.

Consider a pair $\mu, \eta \in \mathcal{M}_1(Y)$, we denote the {\bf transport plan} between $\mu$ and $\eta$ by $\Gamma(\mu, \eta)$ (also known as the {\bf joining} between $\mu$ and $\eta$). In here, we use the definition of {\bf Wasserstein metric} on $\mathcal{M}_1(Y)$ (w.r.t. the metric $d$) obtained via {\bf Kantorovich duality}, which is given by the expression
\begin{equation}
\label{WMKD}
W_d(\mu, \eta) = \sup\Bigl\{ \int_{X_{A, I}} \psi d(\mu - \eta) : \mathrm{Lip}(\psi) \leq 1 \text{ and } \|\psi\|_\infty \leq \mathrm{diam}(Y) \Bigr\} \;,
\end{equation}
where $\mathrm{diam}(Y) := \sup\{d(x, y) : x, y \in Y\}$ is the {\bf diameter} of $Y$ w.r.t. $d$.

Fix a compact metric space $(M, d_M)$, we denote by $X := M^{\mathbb{N}}$ the set of {\bf sequences} taking values into $M$. It is well known that $X$ also results in a compact metric space when it is equipped with the metric $d_X(x, y) := \sum_{n=1}^{\infty} \frac{d_M(x_n, y_n)}{2^n}$. Actually, we can assume w.l.o.g. $\mathrm{diam}(M) = 1$ which implies $\mathrm{diam}(X) = 1$. As usual, the {\bf shift map} is defined by the function $\sigma : X \to X$ satisfying the equation $\sigma((x_n)_{n=1}^{\infty}) := (x_{n+1})_{n=1}^{\infty}$. 

In here we consider a class of one-dimensional lattice systems with markovian structure presented in \cite{MR3194082}. Several properties of the Gibbs states in that context were studied in \cite{zbMATH07249339, zbMATH07434854}. Fixing a continuous map $A : M \times M \to \mathbb{R}$ and a compact set $I \subset \mathbb{R}$, such that, $A(M \times M) \cap I \neq \emptyset$, we say that a sequence $(x_n)_{n=1}^{\infty}$ is {\bf admissible}, when $A(x_n, x_{n+1}) \in I$ for each $n \in \mathbb{N}$. We denote by $X_{A, I} \subset X$ the set of {\bf admissible sequences}. In fact, it is not difficult to check that $X_{A, I}$ is a compact subshift of $X$ (see \cite{MR3194082, zbMATH07434854} for details about the proof of this claim). Besides that, we say that $\omega := x_1\; ...\; x_k \in M^k$ is an {\bf admissible word} when $A(x_n, x_{n+1}) \in I$ for each $n \in \{1, ... , k-1\}$.

Given $b \in M$, we define $s(b) := \{a \in M : A(a, b) \in I\}$. It is not difficult to verify that $s(b) \subset M$ is a compact set for each $b \in M$. So, taking $\mathcal{K}(M)$ as the set of {\bf compact subsets of $M$} equipped with the {\bf Hausdorff metric}, we can check that the map $b \in M \mapsto s(b) \in \mathcal{K}(M)$ is continuous (see for details \cite{MR3194082, zbMATH07434854}). 

Hereafter, we assume that the map $s$ is {\bf locally constant}. Then, for any $z \in X_{A, I}$ there is $\epsilon_z > 0$, such that, each $y \in X_{A, I}$ with $y_1 \in (z_1 - \epsilon_z, z_1 + \epsilon_z)$ satisfies $s(y_1) = s(z_1)$. Define $V_z := \{y \in X_{A, I} : y_1 \in (z_1 - \epsilon_z, z_1 + \epsilon_z)\}$. Then, $V_z$ is a non-empty and open subset of $X_{A, I}$ and $X_{A, I} \subset \bigcup_{z \in X_{A, I}} V_z$. So, by compactness, there is a finite collection $\{z^1, ... , z^{n_0}\} \subset X_{A, I}$, such that, 
\begin{equation}
\label{FC}
X_{A, I} \subset \bigcup_{i=1}^{n_0} V_{z^i} \;.
\end{equation}

Fixing a potential $\varphi \in \mathrm{C}(X_{A, I})$ and an {\bf a priori} measure $\nu \in \mathcal{M}_1(M)$, we define the {\bf Ruelle operator} associated to $\varphi$, as the map sending $\psi \in \mathrm{C}(X_{A, I})$ into $\mathcal{L}_\varphi(\psi) \in \mathrm{C}(X_{A, I})$ given by $\mathcal{L}_\varphi(\psi)(x) := \int_{s(x_1)}e^{\varphi(ax)}\psi(ax) d\nu(a)$, where $ax := (a, x_1, x_2, ... )$. Besides that, we define the {\bf dual Ruelle operator} as the one assigning to each $\mu \in \mathcal{M}(X_{A, I})$ the measure $\mathcal{L}^*_\varphi(\mu) \in \mathcal{M}(X_{A, I})$ satisfying for each $\psi \in \mathrm{C}(X_{A, I})$ the expression $\int_{X_{A, I}} \psi d(\mathcal{L}^*_\varphi(\mu)) := \int_{X_{A, I}} \mathcal{L}_\varphi(\psi) d\mu$.

We say that $\varphi \in \mathrm{C}(X_{A, I})$ is a {\bf normalized potential} when $\mathcal{L}_\varphi({\bf 1}) = {\bf 1}$. In this case the operator $\mathcal{L}^*_\varphi$ preserves the set $\mathcal{M}_1(X_{A, I})$. Moreover, since $|\mathcal{L}_\varphi(\psi)(x)| \leq \|\psi\|_\infty\|\mathcal{L}_\varphi({\bf 1})\|_\infty$ for each $x \in X_{A, I}$, we obtain that $\|\mathcal{L}_\varphi\|_{op} = 1$ which also implies that the {\bf spectral radius} of $\mathcal{L}_\varphi$ is equal to $1$. 

Given a normalized potential $\varphi$, we say that $\mu_\varphi \in \mathcal{M}_\sigma(X_{A, I})$ is a {\bf Gibbs state} for $\varphi$ when it satisfies $\mathcal{L}^*_\varphi(\mu_\varphi) = \mu_\varphi$. Actually, the existence and some properties of the Gibbs states in this setting were presented via Ruelle-Perron-Frobenius theory in \cite{zbMATH07434854} (see also \cite{zbMATH07249339} for properties of the pressure map). 

Throughout the paper we assume that $\varphi$ is a normalized potential belonging to $\mathrm{Lip}(X_{A, I})$ and $\sigma : X_{A, I} \to X_{A, I}$ is {\bf topologically mixing}, i.e., for any open set $U \subset X_{A, I}$, there is $p \in \mathbb{N}$ such that $\sigma^p(U) = X_{A, I}$.

Given $k \in \mathbb{N}$, we define $\Pi_k : X_{A, I} \to M^k$ by $\Pi_k(x) = x_1\; ...\; x_k$. Using the above, we say that $C_k \subset X_{A, I}$ is a {\bf $k$-cylindrical set} associated to $B_1 \times ... \times B_k$, with $B_j \in \mathcal{B}_M$ for each $j \in \{1, ... , k\}$, when $C_k := \Pi_k^{-1}(B_1 \times ... \times B_k)$. Besides that, we say that the $k$-cylindrical set $C_k$ is a {\bf $k$-cylinder}, when $B_j \subset M$ is an open set for any $j$.
Note that the definition of $k$-cylindrical set presented here is a particular case of the one presented in \cite{MR0133846} and agrees with the definitions of $k$-cylinders given in the contexts of finite Markov shifts (see for instance \cite{MR1085356}) and the so called XY-models (see \cite{MR2864625, MR3377291} for details).

A measure $\mu \in \mathcal{M}_1(X_{A, I})$ is {\bf Gibbs-Bowen} for the normalized potential $\varphi$, when there is a constant $C > 1$, such that, for any $m \in \mathbb{N}$, each $m$-cylinder $C_m$ and all $\widetilde{x} \in C_m$, we have
\begin{equation}
\label{GBP}
C^{-1} \leq \frac{\mu(C_m)}{e^{S_m \varphi(\widetilde{x})}} \leq C \;,
\end{equation}
where $S_m\varphi := \sum_{j=0}^{m-1} \varphi \circ \sigma^j$ is the {\bf $m$-th ergodic sum} of $\varphi$.

\section{Uniqueness of the Gibbs state}
\label{contraction-section}

Our main goal in this section is to prove Theorem \ref{theorem-FP}, which states that the dual Ruelle operator $\mathcal{L}^*_\varphi$ has a unique fixed point in $\mathcal{M}_\sigma(X_{A, I})$ when $\varphi$ is a normalized potential. Following the technique proposed in \cite{zbMATH06717808}, we define a suitable metric $D_X$ on the space of sequences $X$ which also induces the product topology, such that, the operator $\mathcal{L}^*_\varphi$ is a contraction on $\mathcal{M}_1(X_{A, I})$ when it is equipped with the Wasserstein metric $W_{D_X}$. Throughout the section we fix a finite collection $\{z^1, ... , z^{n_0}\} \subset X_{A, I}$ satisfying the expression in \eqref{FC}. 

Consider $\psi \in \mathrm{Lip}(X_{A, I})$ and $m \in \mathbb{N}$, by the Ionescu-Marinescu inequality (see proof of Theorem 2 in \cite{MR3377291} for details), it follows that any pair $x, y \in V_{z^i}$, with $i \in \{1, ... , n_0\}$, satisfy 
\begin{equation}
\label{LY}
\Bigl| \mathcal{L}^m_\varphi(\psi)(x) - \mathcal{L}^m_\varphi(\psi)(y) \Bigr| \leq \Bigl( \frac{1}{2^m}\mathrm{Lip}(\psi) + \Bigl( \sum_{k=1}^{m}\frac{1}{2^k} \Bigr)\mathrm{Lip}(e^\varphi)\|\psi\|_\infty \Bigr) d_X(x, y) \;. 
\end{equation}

Then, for any $x, y \in X_{A, I}$, we have $\bigl| \mathcal{L}^m_\varphi(\psi)(x) - \mathcal{L}^m_\varphi(\psi)(y) \bigr| \leq K_md_X(x, y)$, for the constants $K_m := \max\Bigl\{ \frac{\mathrm{diam}(X_{A, I})}{\kappa},\; \frac{1}{2^m}\mathrm{Lip}(\psi) + \Bigl( \sum_{k=1}^{m}\frac{1}{2^k} \Bigr)\mathrm{Lip}(e^\varphi)\|\psi\|_\infty \Bigr\}$ and $\kappa := \min\bigl\{ \frac{\mathrm{diam}(V_{z^i})}{2} : i = 1, ... , n_0 \bigr\}$. 

Assume $\mathrm{Lip}(e^\varphi) > 0$, for each $m \in \mathbb{N}$, we consider the values $\alpha_m := \frac{1}{2^{m+1}}$ and $c_m := \Bigl( \sum_{k=1}^{m}\frac{1}{2^k} \Bigr)\mathrm{Lip}(e^\varphi)$. So, fixing $0 < \delta < \lim_{m \to \infty}\frac{1 - \alpha_m}{2c_m}$, we define a new metric $D_X$ on the space $X$, given by
\begin{equation}
\label{BM}
D_X(x, y) := \min\{1, \delta^{-1}d_X(x, y)\} \;.
\end{equation}

Since $D_X$ also induces the product topology on $X$, the Wasserstein metric $W_{D_X}$ induces the weak* topology on the set $\mathcal{M}_1(X_{A, I})$. Hereafter, for ease of computation, we assume $\mathrm{Lip}(e^\varphi) \geq 1$ which implies $\delta < \frac{1}{2}$. 

For each $x \in X_{A, I}$ and any $m \in \mathbb{N}$, the measure $(\mathcal{L}^m_\varphi)^*(\delta_x) \in \mathcal{M}_1(X_{A, I})$ is given by the expression $(\mathcal{L}^m_\varphi)^*(\delta_x) = \int_{s(a_{m-1})} \cdots \int_{s(x_1)}\delta_{a^mx}e^{S_m\varphi(a^mx)}d\nu^m(a^m)$, with $\nu^m := \otimes_{j=1}^m \nu$ the {\bf product measure} and $a^m x := (a_m, ... , a_1, x_1,...) \in X_{A, I}$ because $a_i \in s(a_{i-1})$ for $i = 2, ... , m$ and $a_1 \in s(x_1)$.

Now we want to show that for $m \in \mathbb{N}$ large enough, the operator $(\mathcal{L}^m_\varphi)^*$ is a contraction on the set of Dirac measures on $X$ when it is equipped with the metric $W_{D_X}$. Moreover, we show that the contraction constant $\alpha \in (0, 1)$ depends exclusively on the potential $\varphi$.

\begin{lemma}
\label{lemma-CWMD}
Consider $\alpha := \max\bigl\{ 1 - \frac{e^{-\mathrm{Lip}(\varphi)}}{2}, \frac{3}{4} \bigr\}$. Then, there is $m_1 \in \mathbb{N}$, such that, any $x, y \in X_{A, I}$ and $m \geq m_1$ satisfy 
\begin{equation*}
W_{D_X}\bigl((\mathcal{L}^m_\varphi)^*(\delta_x), (\mathcal{L}^m_\varphi)^*(\delta_y)\bigr) \leq \alpha W_{D_X}(\delta_x, \delta_y) \;.
\end{equation*}
\end{lemma}
\begin{proof}
- First we consider the case $D_X(x, y) < 1$. 

By \eqref{BM}, we have $D_X(x, y) = \delta^{-1}d_X(x, y)$ which implies $\mathrm{Lip}_{D_X}(\psi) = \delta\mathrm{Lip}(\psi)$ for any $\psi$ belonging to $\mathrm{Lip}(X_{A, I})$, where $\mathrm{Lip}_{D_X}(\psi)$ is the Lipschitz constant of $\psi$ w.r.t. $D_X$. By the former assumption we obtain that $d_X(x, y) < \delta < \frac{1}{2}$. In particular, it follows that $x_1 = y_1$ which also implies that $s(x_1) = s(y_1)$. Therefore, by \eqref{WMKD}, the following expression holds true
\begin{align*}
&W_{D_X}\bigl((\mathcal{L}^m_\varphi)^*(\delta_x), (\mathcal{L}^m_\varphi)^*(\delta_y)\bigr) \\
&= \sup\{|\mathcal{L}^m_\varphi(\psi)(x) - \mathcal{L}^m_\varphi(\psi)(y)| : \mathrm{Lip}(\psi) \leq \delta^{-1} \text{ and } \|\psi\|_\infty \leq 1\}\;.
\end{align*}

Besides that, since $\delta < \lim_{n \to \infty}\frac{1 - \alpha_m}{2c_m}$, there is $m_0 \in \mathbb{N}$, such that any $m \geq m_0$ satisfies $\delta < \frac{1 - \alpha_m}{2c_m}$. By the above and \eqref{LY}, following a standard procedure (see for instance \cite{MR3718407, zbMATH06717808}), we obtain that each $\psi \in \mathrm{Lip}(X_{A, I})$, with $\mathrm{Lip}(\psi) \leq \delta^{-1}$ and $\|\psi\|_\infty \leq 1$, satisfies
\begin{equation*}
|\mathcal{L}^m_\varphi(\psi)(x) - \mathcal{L}^m_\varphi(\psi)(y)|
\leq (\alpha_m \delta^{-1} + c_m)d_X(x, y)
< \frac{3}{4}\delta^{-1} d_X(x, y)
\leq \alpha D_X(x, y) \;.
\end{equation*}

Therefore, we have $W_{D_X}\bigl((\mathcal{L}^m_\varphi)^*(\delta_x), (\mathcal{L}^m_\varphi)^*(\delta_y)\bigr) < \alpha D_X(x, y)$, for $m \geq m_0$, such as we wanted to prove.

- Now we consider the case $D_X(x, y) = 1$. 

Given an admissible word $a^m := a_m\; ...\; a_1$, we use the notation $[a^m]$ for the $m$-cylindrical set associated to the Borel set $\{a_m\} \times ... \times \{a_1\}$. Besides that, for each $x \in X_{A, I}$, we define $\alpha^x(a^m) := {\bf 1}_{\sigma^m([a^m])}(x)\inf\{e^{S_m\varphi(z)} : z \in X_{A, I}\}$ and $\beta^{x}(a^m) := e^{S_m\varphi(a^m x)} - \alpha^x(a^m)$. So, we have $(\mathcal{L}^m_\varphi)^*(\delta_x) = \mu^x_m + \eta^x_m$, where 
\begin{equation*}
\mu^x_m := \int_{s(a_{m-1})} \cdots \int_{s(x_1)} \delta_{a^m x}\alpha^x(a^m)d\nu^m(a^m)
\end{equation*}
and
\begin{equation*}
\eta^x_m := \int_{s(a_{m-1})} \cdots \int_{s(x_1)} \delta_{a^m x}\beta^x(a^m)d\nu^m(a^m) \;.
\end{equation*}

Given a pair $x, y \in X_{A, I}$, define $R^{x, y}_m \in \mathcal{M}(X_{A, I} \times X_{A, I})$ by the expression
\begin{align*}
R^{x, y}_m 
:=& \int_{s(a_{m-1})} \cdots \int_{s(x_1)} {\bf 1}_{\sigma^m([a^m])}(y)\delta_{(a^m x, a^m y)}\alpha^x(a^m)d\nu^m(a^m) \\
&+ \frac{\eta^x_m \otimes \eta^y_m}{\max\{\eta^x_m(X_{A, I}), \eta^y_m(X_{A, I})\}} \;.
\end{align*}

It is not difficult to check that $R^{x, y}_m = R^{y, x}_m$. Moreover, since we are assuming that $\varphi$ is a normalized potential, it follows that $\alpha^x(a^m) \leq 1$ for each $x \in X_{A, I}$, which implies that $R^{x, y}_m(X_{A, I} \times X_{A, I}) \leq 1$. 

Besides that, any $E \in \mathcal{B}_{X_{A, I}}$ satisfies $R^{x, y}_m(E \times X_{A, I}) \leq (\mathcal{L}^m_\varphi)^*(\delta_x)(E)$ and $R^{x, y}_m(X_{A, I} \times E) \leq (\mathcal{L}^m_\varphi)^*(\delta_y)(E)$. So, there exists $Q^{x, y}_m \in \mathcal{M}(X_{A, I} \times X_{A, I})$, such that, $\Pi^{x, y}_m := R^{x, y}_m + Q^{x, y}_m \in \Gamma\bigl((\mathcal{L}^m_\varphi)^*(\delta_x), (\mathcal{L}^m_\varphi)^*(\delta_y)\bigr)$. Given $k \in \mathbb{N}$, we define the set $\Delta_k := \{(x', y') \in X_{A, I} \times X_{A, I} : x'_1 = y'_1, ... ,x'_k = y'_k \}$. Then, any $m \geq k$ satisfies
\begin{align*}
\Pi^{x, y}_m(\Delta_k) 
&\geq \int_{s(a_{m-1})} \cdots \int_{s(x_1)} \delta_{(a^m x, a^m y)}(\Delta_k){\bf 1}_{\sigma^m([a^m])}(y)\alpha^x(a^m)d\nu^m(a^m) \\
&= \inf\{e^{S_m\varphi(y) - S_m\varphi(a^m x)} : y \in X_{A, I}\} \mathcal{L}_\varphi({\bf 1})(x)
\geq e^{-\mathrm{Lip}(\varphi)} \;.
\end{align*}

In particular, taking the supremum on $\Gamma\bigl((\mathcal{L}^m_\varphi)^*(\delta_x), (\mathcal{L}^m_\varphi)^*(\delta_y)\bigr)$, we obtain that $\sup\{\Pi(\Delta_k) : \Pi \in \Gamma\bigl((\mathcal{L}^m_\varphi)^*(\delta_x), (\mathcal{L}^m_\varphi)^*(\delta_y)\bigr) \} \geq e^{-\mathrm{Lip}(\varphi)}$ for every $m \geq k$.

Then, since $d_X(x', y') \leq \frac{\delta}{2}$ implies that $D_X(x', y') = \delta^{-1}d_X(x', y') \leq \frac{1}{2}$, fixing $m_1 \geq m_0$, such that, $\frac{1}{2^{m_1}} \leq \frac{\delta}{2}$, it follows that (see for details \cite{MR3718407, zbMATH06717808}), each $m \geq m_1$ satisfies
\begin{align*}
W_{D_X}\bigl((\mathcal{L}^m_\varphi)^*(\delta_x), (\mathcal{L}^m_\varphi)^*(\delta_y)\bigr) 
&\leq \frac{1}{2}\int_{\Delta_{m_1}} d\Pi^{x, y}_m + \Bigl(1 - \int_{\Delta_{m_1}} d\Pi^{x, y}_m\Bigr) \\
&= 1 - \frac{1}{2}\Pi^{x, y}_m(\Delta_{m_1}) 
\leq \alpha = \alpha D_X(x, y) \;.
\end{align*}

Then, we have $W_{D_X}\bigl((\mathcal{L}^m_\varphi)^*(\delta_x), (\mathcal{L}^m_\varphi)^*(\delta_y)\bigr) < \alpha D_X(x, y)$, for $m \geq m_1$. By the above, our result holds true.
\end{proof}

Given $\mu, \eta \in \mathcal{M}_1(X_{A, I})$, $Q \in \Gamma(\mu, \eta)$ and $\Pi_m^{x, y} \in \Gamma\bigl((\mathcal{L}^m_\varphi)^*(\delta_x), (\mathcal{L}^m_\varphi)^*(\delta_y)\bigr)$, with $x, y \in X_{A, I}$, it follows that the measure $S \in \mathcal{M}_1(X_{A, I} \times X_{A, I})$ satisfying the expression $dS(x', y') := \int_{X_{A, I} \times X_{A, I}} d\Pi_m^{x, y}(x', y')dQ(x, y)$ belongs to $\Gamma\bigl((\mathcal{L}^m_\varphi)^*(\mu), (\mathcal{L}^m_\varphi)^*(\eta)\bigr)$. The above, because any pair $\psi_1, \psi_2 \in \mathrm{C}(X_{A, I})$ satisfy
\begin{equation*}
\int_{X_{A, I} \times X_{A, I}} (\psi_1(x) + \psi_2(y)) dS(x, y) = \int_{X_{A, I}} \psi_1 d\bigl( (\mathcal{L}^m_\varphi)^*(\mu) \bigr) + \int_{X_{A, I}} \psi_2 d\bigl( (\mathcal{L}^m_\varphi)^*(\eta) \bigr) \;.
\end{equation*}

Then, taking $Q \in \Gamma(\mu, \eta)$ as a $D_X$-optimal plan, by Lemma \ref{lemma-CWMD}, it follows that the operator $(\mathcal{L}^m_\varphi)^*$, with $m \geq m_1$, is a contraction on $\mathcal{M}_1(X_{A, I})$ for the Wasserstein metric $W_{D_X}$ (see for details \cite{MR3718407, zbMATH06717808}), i.e., there exists $\beta \in (0, 1)$, such that, any pair $\mu, \eta \in \mathcal{M}_1(X_{A, I})$ satisfy
\begin{equation}
\label{WM-contraction}
W_{D_X}\bigl((\mathcal{L}^m_\varphi)^*(\mu), (\mathcal{L}^m_\varphi)^*(\eta)\bigr) \leq \beta W_{D_X}(\mu, \eta) \;.
\end{equation}

The main result of this section guarantees uniqueness of the Gibbs state associated to $\varphi$. The statement of the result is the following one. 

\begin{theorem}
\label{theorem-FP}
There is a unique Gibbs state $\mu_\varphi \in \mathcal{M}_\sigma(X_{A, I})$ for the normalized potential $\varphi$. Furthermore, any $\mu \in \mathcal{M}_1(X_{A, I})$ satisfies $\lim_{m \to \infty}(\mathcal{L}^m_\varphi)^*(\mu) = \mu_\varphi$ in the weak* topology.  
\end{theorem}
\begin{proof}
First note that $\mathcal{M}_1(X_{A, I})$ is a compact metric space when it is equipped with the Wasserstein metric $W_{D_X}$. So, by \eqref{WM-contraction} and the Banach contraction principle, we can guarantee existence of a unique fixed point $\mu_\varphi \in \mathcal{M}_1(X_{A, I})$ for the operator $\mathcal{L}^*_\varphi$ satisfying $\lim_{m \to \infty}(\mathcal{L}^m_\varphi)^*(\mu) = \mu_\varphi$ for any $\mu \in \mathcal{M}_1(X_{A, I})$ (where the limit is taken in the weak* topology). Besides that, since we are assuming $\mathcal{L}_\varphi({\bf 1}) = {\bf 1}$, it follows that any fixed point of $\mathcal{L}^*_\varphi$ belongs to $\mathcal{M}_\sigma(X_{A, I})$ (see \cite{MR3377291} for details), which implies $\mu_\varphi \in \mathcal{M}_\sigma(X_{A, I})$.
\end{proof}

\section{Statistical properties of the Gibbs state}
\label{statistical-properties-section}

In this section we present some results about the statistical behavior of the Gibbs state $\mu_\varphi \in \mathcal{M}_\sigma(X_{A, I})$ given by Theorem \ref{theorem-FP}. The first one of the properties that we show here is that $\mu_\varphi$ is Gibbs-Bowen. In addition, we prove that the Gibbs state satisfies exponential decay of correlations, which also implies that it is a mixing invariant probability measure and, at the end of the section, we show that $\mu_\varphi$ satisfies a central limit theorem. 

\begin{theorem}
\label{theorem-GB}
The probability measure $\mu_\varphi$ is Gibbs-Bowen. 
\end{theorem}
\begin{proof}
Given $m \in \mathbb{N}$, let $C_m$ be a $m$-cylinder and $C_{m+1}$ be a $(m+1)$-cylinder. Observe that $|S_m\varphi(x) - S_m\varphi(y)| \leq \Bigl(\sum_{j=1}^m \frac{1}{2^j}\Bigr)\mathrm{Lip}(\varphi)d_X(x, y) \leq \mathrm{Lip}(\varphi)$ for any pair $x, y \in X_{A, I}$.

By compactness, there is $\phi_m \in \mathrm{C}(X_{A, I})$ satisfying ${\bf 1}_{C_{m+1}} \leq \phi_m \leq {\bf 1}_{C_m}$. Therefore, since $(\mathcal{L}^{m+1}_\varphi)^*(\mu_\varphi) = \mu_\varphi$, by a standard argument, it follows that $\mu_\varphi(C_{m+1}) \leq e^{S_{m+1}\varphi(\widetilde{x}) + \mathrm{Lip}(\varphi)}$ for any $\widetilde{x} \in C_m$.

On the other hand, given $x \in X_{A, I}$ and an admissible word $a^{m+p}$ such that $a^{m+p}x \in X_{A, I}$, we have ${\bf 1}_{\sigma^{m+p}(C_{m+1})}(x) = 0$ when ${\bf 1}_{C_{m+1}}(a^{m+p} x) = 0$. By the above, we can assure that ${\bf 1}_{C_{m+1}}(a^{m+p} x) \geq {\bf 1}_{\sigma^{m+p}(C_{m+1})}(x)$ for each $x \in X_{A, I}$. Denote by $I := \inf\{e^{\varphi(x)} : x \in X_{A, I}\} > 0$, since $X_{A, I}$ is topologically mixing, there is $p = p(m) \in \mathbb{N}$ such that $\sigma^{m+p}({\bf 1}_{C_{m+1}}) = X_{A, I}$. Besides that, $(\mathcal{L}^{m+p}_\varphi)^*(\mu_\varphi) = \mu_\varphi$, which implies that $\mu_\varphi(C_m) \geq I^p e^{S_m\varphi(\widetilde{x}) - \mathrm{Lip}(\varphi)}$ for any $\widetilde{x} \in C_{m+1}$. 

So, taking $C := \max\{e^{\mathrm{Lip}(\varphi)}, e^{\mathrm{Lip}(\varphi)}I^{-p}\}$ we obtain that the expression in \eqref{GBP} holds true for $\mu_\varphi$, such as we wanted to prove.
\end{proof}

Our next goal is to prove that the Gibbs state $\mu_\varphi$ has exponential decay of correlations between any pair of potentials $\phi \circ \sigma^m, \psi \in \mathrm{Lip}(X_{A, I})$, with $m \in \mathbb{N}$.

\begin{theorem}
\label{theorem-DC}
There exists $\Lambda \in (0, 1)$, such that, for any pair $\phi, \psi \in \mathrm{Lip}(X_{A, I})$ and each $m \in \mathbb{N}$ we have 
\begin{equation*}
\Bigl| \int_{X_{A, I}} (\phi \circ \sigma^m)\psi d\mu_\varphi - \int_{X_{A, I}} \phi d\mu_\varphi \int_{X_{A, I}} \psi d\mu_\varphi \Bigl| \leq C_{\phi, \psi}\Lambda^m \;.
\end{equation*}
\end{theorem}
\begin{proof}
Consider $V := \bigl\{\widetilde{\psi} \in \mathrm{Lip}(X_{A, I}) : \int_{X_{A, I}} \widetilde{\psi} d\mu_\varphi = 0 \bigr\}$. Since $\mathcal{L}_\varphi({\bf 1}) = {\bf 1}$, it is not difficult to check that $\mathrm{Lip}(X_{A, I}) = \mathrm{span}\{{\bf 1}\} \oplus V$ and $\mathcal{L}_\varphi(V) \subset V$. Besides that, we know that the sequence $(\mathcal{L}^m_\varphi(\widetilde{\psi}))_{m=1}^\infty$ converges to ${\bf 0}$ uniformly in the norm $\| \cdot \|_\infty$ for any $\widetilde{\psi} \in V$  (see the proof of Theorem 1 in \cite{zbMATH07434854} for details). By the above, it follows that any pair $\phi, \psi \in \mathrm{Lip}(X_{A, I})$ satisfy
 \begin{equation*}
\Bigl| \int_{X_{A, I}}(\phi \circ \sigma^m)\widetilde{\psi}d\mu_\varphi \Bigr| = \Bigl| \int_{X_{A, I}} \phi \mathcal{L}^m_\varphi(\widetilde{\psi})d\mu_\varphi \Bigr| \leq \| \mathcal{L}_\varphi|_V \|^m_{op} \| \phi \|_1 \|\widetilde{\psi}\|_\infty \;,
\end{equation*}
where $\widetilde{\psi} = \psi - \int_{X_{A, I}} \psi d\mu_\varphi$. Since $\mu_\varphi \in \mathcal{M}_\sigma(X_{A, I})$, taking $\Lambda := \| \mathcal{L}_\varphi|_V \|_{op}$ and $C_{\phi, \psi} = \| \phi \|_1 \|\widetilde{\psi}\|_\infty$, our result holds true.
\end{proof}

\begin{corollary}
Any pair of functions $\phi, \psi \in \mathrm{Lip}(X_{A, I})$ satisfy
\begin{equation*}
\lim_{m \to \infty} \int_{X_{A, I}} (\phi \circ \sigma^m)\psi d\mu_\varphi = \int_{X_{A, I}} \phi d\mu_\varphi \int_{X_{A, I}} \psi d\mu_\varphi \;.
\end{equation*}
\end{corollary}

The last part of this paper is dedicated to prove a central limit theorem for the Gibbs state $\mu_\varphi \in \mathcal{M}_\sigma(X_{A, I})$. In order to get that, we consider the space $L^2(\mu_\varphi)$, the {\bf Koopman operator} $\mathcal{U} : L^2(\mu_\varphi) \to L^2(\mu_\varphi)$ given by the expression $\mathcal{U}(\psi) := \psi \circ \sigma$ and the {\bf adjoint} $\mathcal{U}^*$ of the Koopman operator as the one satisfying $\int_{X_{A, I}} \psi_1 \mathcal{U}(\psi_2) d\mu_\varphi = \int_{X_{A, I}} \mathcal{U}^*(\psi_1)\psi_2 d\mu_\varphi$, for each pair of functions $\psi_1, \psi_2 \in L^2(\mu_\varphi)$. 

We define the {\bf expected value} of $\psi \in L^2(\mu_\varphi)$ by $\mathbb{E}(\psi) := \int_{X_{A, I}}\psi d\mu_\varphi$. Given an {\bf algebra} $\mathcal{A} \subset \mathcal{B}_{X_{A, I}}$, we let $L^2(\mathcal{A})$ be the set of {\bf $\mathcal{A}$-measurable} functions in $L^2(\mu_\varphi)$ and, for any potential $\psi \in L^2(\mu_\varphi)$, we define the {\bf conditional expectation} of $\psi$ w.r.t. the algebra $\mathcal{A}$, as the map $\mathbb{E}(\psi | \mathcal{A}) : X_{A, I} \to \mathbb{R}$, such that, each $A \in \mathcal{A}$ satisfies $\int_A \mathbb{E}(\psi | \mathcal{A}) d\mu_\varphi = \int_A \psi d\mu_\varphi$.

It is not difficult to check that $\mathbb{E}(\mathbb{E}(\psi | \mathcal{A}) | \mathcal{A}) = \mathbb{E}(\psi | \mathcal{A})$. Therefore, $\mathbb{E}(\psi | \mathcal{A})$ is the orthogonal projection of $\psi$ into $L^2(\mathcal{A})$. Given $\psi \in L^2(\mu_\varphi)$, consider $\widetilde{\psi} := \psi - \mathbb{E}(\psi)$, we define the {\bf variance} of $\psi$ by $S^2(\psi) := \mathbb{E}(\widetilde{\psi}^2) \geq 0$. The value $S(\psi) := \sqrt{S^2(\psi)}$ is known as the {\bf standard deviation} of $\psi \in L^2(\mu_\varphi)$. 

Now we are able to state a central limit theorem in our setting. The proof presented here is a consequence of the decay of correlations (see Theorem \ref{theorem-DC}). Actually, we follow a similar argument to the one in \cite{zbMATH03337280} (see also \cite{Via97} for a detailed proof). We show existence of a Martingale difference of random variables $(\mathcal{U}^m(\rho))_{m=0}^\infty$ (w.r.t. a suitable nested sequence of sub-algebras of $\mathcal{B}_{X_{A, I}}$), such that $S^2(\psi) = S^2(\rho)$ (with $\rho$ to be defined below). So, our result follows from the central limit theorem for Martingale differences (see for details \cite{MR1700749, zbMATH03200207}). 

In \cite{zbMATH03337280}, the author assumes that $\mu \in \mathcal{M}_1(X_{A, I})$ and $\psi \in L^2(\mu)$ satisfy a suitable condition of convergence for the means of the iterates of $\mathcal{U}$ w.r.t. an orthogonal decomposition of $L^2(\mu)$ (actually, \eqref{cohom} is a particular case of the condition stated in \cite{zbMATH03337280} when $T$ is an endomorphism). Then, he proves that $0 < S(\psi) < \infty$ and the central limit theorem holds true. Moreover, he also concludes that $S^2(\psi) = S^2(\rho)$ in his setting, which appears as a consequence of the decay of correlations in our approach.

The statement of the last theorem of this paper is the following one.

\begin{theorem}
\label{theorem-CL}
Consider a function $\psi \in \mathrm{Lip}(X_{A, I})$. Then, $S(\psi) < \infty$ and the following properties hold true:
\begin{enumerate}[i)] 
\item When $S(\psi) > 0$, we have 
\begin{equation*}
\lim_{m \to \infty} \mu_\varphi\Bigl( x \in X_{A, I} : \frac{1}{\sqrt{m}}\sum_{j=0}^{m-1}\mathcal{U}^j(\widetilde{\psi})(x) \in (a, b) \Bigr) = \frac{1}{S(\psi)\sqrt{2\pi}}\int_a^b e^{-\frac{t^2}{2 S(\psi)}} dt\;;
\end{equation*}
\item $S(\psi) = 0$ if, and only if, $\psi = u - \mathcal{U}(u)$ for some $u \in L^2(\mu_\varphi)$.
\end{enumerate}
\end{theorem}
\begin{proof}
For $m \in \mathbb{N} \cup \{0\}$, consider the algebra $\mathcal{B}_m := \sigma^{-m}(\mathcal{B}_{X_{A, I}})$. It is easy to check that $\mathcal{B}_{X_{A, I}} \supset \mathcal{B}_1 \supset \cdots \supset \mathcal{B}_m \supset \cdots $. So, defining $\mathcal{B}_\infty := \bigcap_{m=0}^\infty \mathcal{B}_m$, we obtain $L^2(\mu_\varphi) \supset L^2(\mathcal{B}_1) \supset \cdots \supset L^2(\mathcal{B}_m) \supset \cdots \supset L^2(\mathcal{B}_\infty)$. Furthermore, we have $\mathcal{U}(L^2(\mathcal{B}_m)) = L^2(\mathcal{B}_{m+1})$, $\mathcal{U}^*(L^2(\mathcal{B}_{m+1})) = L^2(\mathcal{B}_m)$ and, by the Riesz representation theorem, we obtain that
\begin{equation*}
\| \mathbb{E}(\widetilde{\psi} | \mathcal{B}_m ) \|_2
\leq \sup\Bigl\{ \mathbb{E}(\mathcal{U}^m(\phi_m) \widetilde{\psi}) :\; \phi_m \in \mathrm{Lip}(X_{A, I}),\, \|\phi_m\|_1 \leq 1 \Bigr\}
\leq C_{\phi_m, \psi} \Lambda^m \;,
\end{equation*}
where the last one of the inequalities is given by Theorem \ref{theorem-DC}, i.e., $\Lambda \in (0, 1)$ and $C_{\phi_m, \psi} = \|\phi_m\|_1\|\widetilde{\psi}\|_\infty \leq \|\widetilde{\psi}\|_\infty$. By the above
\begin{equation}
\label{zeta-2}
\sum_{m = 1}^\infty \|\mathbb{E}(\widetilde{\psi} | \mathcal{B}_m )\|_2^2 \leq \sum_{m = 1}^\infty C_{\phi_m, \psi} \Lambda^m \leq \frac{\|\widetilde{\psi}\|_\infty}{1 - \Lambda} \;.
\end{equation} 

Besides that, since $\mathbb{E}(\widetilde{\psi} | \mathcal{B}_m)$ is the projection of $\widetilde{\psi}$ on $L^2(\mathcal{B}_m)$, it follows that $\|\mathbb{E}(\widetilde{\psi} | \mathcal{B}_m) - \mathbb{E}(\widetilde{\psi} | \mathcal{B}_{m+1})\|_2 
\leq \|\mathbb{E}(\widetilde{\psi} | \mathcal{B}_m)\|_2$. So, by \eqref{zeta-2}, it follows that
\begin{equation}
\label{rho-2}
\sum_{m = 1}^\infty \|\mathbb{E}(\widetilde{\psi} | \mathcal{B}_m) - \mathbb{E}(\widetilde{\psi} | \mathcal{B}_{m+1})\|_2^2 \leq \frac{\|\widetilde{\psi}\|_\infty}{1 - \Lambda} \;.
\end{equation}

Define the maps
\begin{equation}
\label{rho}
\zeta := -\sum_{m=1}^\infty (\mathcal{U}^*)^m(\mathbb{E}(\widetilde{\psi} | \mathcal{B}_m))\;, \ \ \rho := \sum_{m=0}^\infty (\mathcal{U}^*)^m(\mathbb{E}(\widetilde{\psi} | \mathcal{B}_m) - \mathbb{E}(\widetilde{\psi} | \mathcal{B}_{m+1})) \;.
\end{equation}

By \eqref{zeta-2}, it follows that $\zeta \in L^2(\mu_\varphi)$ and, by \eqref{rho-2}, we obtain that $\rho \in L^2(\mu_\varphi)$. So, reordering the series in \eqref{rho}, the expression $\rho = \widetilde{\psi} - \zeta + \mathcal{U}(\zeta)$ holds true. In particular, the above implies that
\begin{equation}
\label{cohom}
\frac{1}{\sqrt{m}} \sum_{j=0}^{m-1}\mathcal{U}^j(\widetilde{\psi}) = \frac{1}{\sqrt{m}} \sum_{j=0}^{m-1}\mathcal{U}^j(\rho) + \frac{1}{\sqrt{m}}(\zeta - \mathcal{U}^m(\zeta)) \;.
\end{equation}

Besides that, since $\zeta \in L^2(\mu_\varphi)$, for any $\delta > 0$ we have
\begin{equation}
\label{CM}
\lim_{m \to \infty} \mu_\varphi\Bigl( x \in X_{A, I} : \frac{1}{\sqrt{m}}| \zeta - \mathcal{U}^m(\zeta) | > \delta \Bigr) = 0 \;.
\end{equation}

On the other hand, since $(\mathcal{U}^*)^m(\mathbb{E}(\widetilde{\psi} | \mathcal{B}_m) - \mathbb{E}(\widetilde{\psi} | \mathcal{B}_{m+1})) \in L^2(\mathcal{B}_1)^{\bot}$ for each $m \in \mathbb{N} \cup \{0\}$, it follows that $\rho \in L^2(\mathcal{B}_1)^{\bot}$. Therefore, $\mathbb{E}(\rho | \mathcal{B}_1) = 0$ and
\begin{equation}
\label{mart}
\mathbb{E}(\mathcal{U}^m(\rho) | \mathcal{B}_{m+1}) = \mathcal{U}^m(\mathbb{E}(\rho | \mathcal{B}_1)) = 0 \;.
\end{equation} 

The above implies that the sequence of random variables $(\mathcal{U}^m(\rho))_{m=0}^\infty$ is a Martingale difference for the nested sequence of algebras $(\mathcal{B}_m)_{m=0}^\infty$. So, the central limit theorem for Martingale differences holds true (see \cite{MR1700749, zbMATH03200207} for details). That is, when we have $0 < S(\rho) < \infty$, it follows that
\begin{equation*}
\lim_{m \to \infty} \mu_\varphi\Bigl( x \in X_{A, I} : \frac{1}{\sqrt{m}}\sum_{j=0}^{m-1}\mathcal{U}^j(\rho)(x) \in (a, b) \Bigr) = \frac{1}{S(\rho)\sqrt{2\pi}}\int_a^b e^{-\frac{t^2}{2 S(\rho)}} dt \;.
\end{equation*} 

Furthermore, since $\mathcal{U}^m(\rho) \in L^2(\mathcal{B}_m)$, it follows that $\mathcal{U}^k(\rho) \in L^2(\mathcal{B}_{m+1})$ for any $k > m$, thus, by \eqref{mart}, we obtain that $\mathbb{E}(\mathcal{U}^k(\rho)\mathcal{U}^m(\rho)) = 0$. The above, joint with \eqref{cohom}, implies that $\|\rho\|_2^2 = \mathbb{E}(\widetilde{\psi}^2) + 2\lim_{n \to \infty} \Bigl(\sum_{j=1}^{n-1} \frac{j}{n}\mathbb{E}(\widetilde{\psi}\mathcal{U}^j(\widetilde{\psi})) \Bigr)$.

Since $\mathbb{E}(\widetilde{\psi}\mathcal{U}^j(\widetilde{\psi})) \leq \|\mathbb{E}(\widetilde{\psi} | \mathcal{B}_j)\|_2\|\widetilde{\psi}\|_2$ for each $j \in \mathbb{N} \cup \{0\}$, it follows that any $k \in \mathbb{N}$ and $n \geq k$ satisfy
\begin{equation*}
\Bigl\| \sum_{j=1}^{n-1} \frac{j}{n}\mathbb{E}(\widetilde{\psi}\mathcal{U}^j(\widetilde{\psi})) \Bigr\|_2
\leq \|\widetilde{\psi}\|_2\Bigl(\sum_{j=1}^{k-1} \frac{j}{n}\|\mathbb{E}(\widetilde{\psi} | \mathcal{B}_j)\|_2 + \sum_{j=k}^{n-1} \|\mathbb{E}(\widetilde{\psi} | \mathcal{B}_j)\|_2 \Bigr) \;.
\end{equation*}

Fix $\epsilon > 0$, by \eqref{zeta-2}, there is $k = k(\epsilon)$, such that, $\sum_{j = k}^\infty \|\mathbb{E}(\widetilde{\psi} | \mathcal{B}_j)\|_2 < \epsilon$ and $n = n(\epsilon, k) > k$, such that, $\epsilon > \frac{k-1}{n} > 0$. In particular
\begin{equation*}
\Bigl\| \sum_{j=1}^{n-1} \frac{j}{m}\mathbb{E}(\widetilde{\psi}\mathcal{U}^j(\widetilde{\psi})) \Bigr\|_2
\leq \|\widetilde{\psi}\|_2\Bigl(\epsilon \sum_{j=1}^{k-1} \|\mathbb{E}(\widetilde{\psi} | \mathcal{B}_j)\|_2 + \epsilon \Bigr) 
< \epsilon \|\widetilde{\psi}\|_2 \Bigl(\frac{\|\widetilde{\psi}\|_\infty}{1 - \Lambda} + 1\Bigr)\;.
\end{equation*}

Then, taking $\epsilon \to 0$, we obtain that $\lim_{n \to \infty} \bigl\| \sum_{j=1}^{n-1} \frac{j}{n}\mathbb{E}(\widetilde{\psi}\mathcal{U}^j(\widetilde{\psi})) \bigr\|_2 = 0$ which implies $\|\rho\|_2^2 = \mathbb{E}(\widetilde{\psi}^2) = S^2(\psi)$. But, $\mathbb{E}(\rho) = \mathbb{E}(\widetilde{\psi} - \zeta + \mathcal{U}(\zeta)) = \mathbb{E}(\widetilde{\psi}) = 0$, so, we have $S^2(\rho) = \mathbb{E}(\rho^2) = \|\rho\|_2^2 = S^2(\psi)$. 

In particular, when $S(\psi) > 0$, by central limit theorem for Martingale differences, it follows that
\begin{equation*}
\lim_{m \to \infty} \mu_\varphi\Bigl( x \in X_{A, I} : \frac{1}{\sqrt{m}}\sum_{j=0}^{m-1}\mathcal{U}^j(\rho)(x) \in (a, b) \Bigr) = \frac{1}{S(\psi)\sqrt{2\pi}}\int_a^b e^{-\frac{t^2}{2 S(\psi)}} dt \;.
\end{equation*}  

- Assume that $S(\psi) > 0$ and for any pair $a, b \in \mathbb{R}$, with $a < b$, define $\Phi(a, b) := \frac{1}{S(\psi)\sqrt{2\pi}}\int_a^b e^{-\frac{t^2}{2 S(\psi)}} dt$. Observe that $(a, b) \mapsto \Phi(a, b)$ is continuous. Then, following a similar procedure that the one in \cite{Via97}, by \eqref{cohom} and \eqref{CM}, we obtain that the limit at item $i)$ of the theorem holds true. 

- Assume that $S(\psi) = 0$. Then, we have $\rho = 0$ which implies $\widetilde{\psi} = \zeta - \mathcal{U}(\zeta)$. Conversely, when $\widetilde{\psi} = u - \mathcal{U}(u)$, it follows that $\zeta = u$ and $\rho = 0$, which implies $S(\psi) = 0$. The above proves item $ii)$ of the theorem.
\end{proof}

\section*{Acknowledgments} The author would to thank to the University of Porto for the hospitality during the development of this paper and to the Foundation for Science and Technology (FCT) Project UIDP/00144/2020 for the financial support.

\end{document}